\documentclass[12pt,a4paper]{article}
\pdfoutput=1 %ArXiv
\usepackage[T1]{fontenc}
\usepackage[utf8]{inputenc}
\usepackage[english]{babel}
\usepackage{authblk}
\usepackage{amsmath,amssymb,amsthm}
\usepackage{graphicx}
\usepackage{color}
\usepackage{hyperref}
\usepackage{todonotes}
\usepackage{amsfonts}
\usepackage{mathrsfs}
\usepackage{url}
\newtheorem{theorem}{Theorem}
\newtheorem{lemma}{Lemma}
\newtheorem{remark}{Remark}
\newtheorem{proposition}{Proposition}

\allowdisplaybreaks
\newcommand{\Expect}{\mathsf{E}}

\begin{document}

\begin{center}
	\textbf{{\Large Structure of the particle population for a branching random walk with a critical reproduction law}}
\end{center}

\begin{center}
	\textit{{\large Daria Balashova, Stanislav Molchanov,  Elena Yarovaya}}
\end{center}

\title{Structure of the particle population for a branching random walk with a critical reproduction law}

\begin{abstract}
We consider a continuous-time symmetric branching random walk on the $d$-dimensional lattice, $d\ge 1$, and assume that at the initial moment there is one particle at every lattice point. Moreover, we assume that the underlying random walk has a finite variance of jumps and the reproduction law is described by a critical Bienamye-Galton-Watson process at every lattice point. We study the structure of the particle subpopulation generated by the initial particle situated at a lattice point $x$. We answer why vanishing of the majority of subpopulations does not affect the convergence to the steady state and leads to clusterization for lattice dimensions $d=1$ and $d=2$.

\medskip

\end{abstract}

\section{Introduction}
\label{intro}

Consider a lattice population model, i.e. a random field $n(t,\cdot)$ of  particles on $\mathbb{Z}^{d}$, $d\ge 1$, where $n(t,y)$ is the number of particles at the point $y \in \mathbb{Z}^{d}$ at the time moment $t\ge 0$. Let $n(0,y)\equiv 1$ for every $y \in \mathbb{Z}^{d}$. The spatio-temporal evolution of the field includes the migration and birth-death processes. As usual, we exclude interaction between particles.

The migration until the first transformation (death or particle splitting) is described by a random walk of particles governed by the generator
\begin{equation} \label{l_eq1}
\mathcal{L}\psi(x)=\kappa\sum\limits_{z\ne 0}[\psi(x+z)-\psi(x)]a(z).
\end{equation}
Here $\kappa >0$ is the diffusion coefficient and $a(z)\ge 0$, $\sum\limits_{z\ne 0}a(z)=-a(0)=1$ is the distribution of the random walk jumps. We assume symmetry $a(z)=a(-z)$ and irreducibility of the random walk, see~\cite{Molchanov Yarovaya 2013}.
%, i.e. for all $z\in\mathbb{Z}^d$ there exists a set of vectors $z_1, z_2,\ldots,z_k \in\mathbb{Z}^d$ such that $z=\sum_{i=1}^k %z_i$ and $a(z_i)\ne 0$ for $i=1,2,\ldots,k$.
Finally, we assume that, for every $\lambda\in \mathbb{R}^{d}$,
\[
\sum_{z\in \mathbb{Z}^{d}}e^{(\lambda, z)}a(z)<\infty.
\]
It means that the tails of $a(z)$ are superexponentially light. Note that the large deviation theorems from~\cite{Molchanov Yarovaya 2013} based on the last condition.

The transition probabilities of the random walk $X(t), t\ge 0$, with the generator $\mathcal{L}$, i.e. $p(t,x,y)=P_x\{X(t)=y\}$, satisfy, see \cite{Gikhman},
the Kolmogorov backward equations
\begin{equation}
\label{eq2}
\frac{\partial p}{\partial t}=\mathcal{L}p,\quad p(0,x,y)=\delta(y-x).
\end{equation}

Consider the Fourier transform
\[
\hat a(k)=\sum\limits_{z\ne 0}e^{i(k,x)}a(z)=\sum\limits_{z\ne 0}\cos(kz)a(z).
\]
Due to the homogeneity of the random walk over the space, we have $p(t,x,y)=p(t,0,y-x)=p(t,0,x-y)$. By applying the Fourier transform to equation \eqref{eq2} we get
\[
p(t,x,y)=\frac{1}{(2\pi)^d}\int\limits_{T^d}e^{-\kappa t(1-\hat a(k)+ik(x-y))} dk,\quad T^d=[-\pi,\pi]^d.
\]
Note that $\hat a(k)$ is a real function and as a result we have the following important inequality:
\begin{equation}\label{l_eq4}
p(t,0,z)\le p(t,0,0)=\frac{1}{(2\pi)^d}\int\limits_{T^d}e^{-\kappa t(1-\hat a(k))}dk.
\end{equation}
In what follows we use the notation: $p(t,z):=p(t,0,z)$ and  $p(t,y-x):=p(t,x,y)$.

We assume that the intensity of jumps $\kappa$ in \eqref{l_eq1} is equal to one, then the local Central Limit Theorem (CLT) from \cite{Molchanov Yarovaya 2013} implies for $|z|<C\sqrt{t}$ that
\begin{equation} \label{l_eq5}
p(t,z)\sim \frac{e^{-\frac{|z|^2}{2t}}}{(2\pi t)^{d/2}}, \quad t\to\infty.
\end{equation}

Each particle of the population dies in the interval ($t,t+dt$) with probability $\mu dt$, where $\mu$ is the mortality rate, or splits into two identical particles with probability $\beta dt$.
We call $\beta$ the birth rate. Further we consider only the critical branching process: $\beta=\mu$ at every lattice point.

The total population $n(t,y)$ at a point $y\in \mathbb{Z}^{d}$ is the sum of independent subpopulations:
\[
n(t,y)  =\sum\limits_{x\in \mathbb{Z}^{d}}n(t,x,y),
\]
where $n(0,x,y)  =\delta(y-x)$, and  $n(0,y) \equiv 1$.

For the moment analysis of the field $n(t,y)$ one can use either the forward Kolmogorov equations for the correlation functions
\[
K_t
(x_1,...,x_m)=\Expect n(t,x_1)...n(t,x_m),
\]
or the backward approach for $n(t,x,y)$ with the following summation over $x$ of the cumulants of subpopulations. Note that the first approach is simpler for continuous models with $\mathbb{R}^d$ instead of $\mathbb{Z}^{d}$, see~\cite{Kondratiev}.
Their method requires one particle in each site and is not applicable to the lattice model. In the latter case the backward approach was developed in~\cite{Molchanov Whitmeyer}.

The papers~\cite{Kondratiev},~\cite{Molchanov Whitmeyer} contain the following results: a branching random walk for $\mu=\beta$ on $\mathbb{Z}^{d}$ (or $\mathbb{R}^d$) converges to a steady state (statistical equilibrium)  if the underlying process $X(t)$ with the generator $\mathcal{L}$ is transient.

\begin{remark}
Theorems 3.1 and 3.2 in~\cite{Kondratiev} contain some additional assumptions about the function $a(x)$, $x \in \mathbb{R}^d$. In the lattice case we have no additional assumptions on $X(t)$, except the transience.
\end{remark}

For the generating function of subpopulation $u_z(t, x, y)=\Expect z^{n(t,x,y)}$
we have the parabolic problem with quadratic non-linearity (see (16) in~\cite{Molchanov Whitmeyer}):
\begin{equation} \label{l_eqr}
\begin{aligned}
\frac{\partial u_z(t, x, y)}{\partial t}  & =\mathcal{L}u_z(t, x, y)+\beta u_z^2(t, x, y)-(\beta+\mu)u_z(t, x, y)+\mu,\\
u_z(0, x, y)  & =\begin{cases}
z, \ x=y,\\
1, \ x\ne y.
\end{cases}
\end{aligned}
\end{equation}
From this last equation in the critical case $\mu=\beta$ we get
\begin{equation} \label{l_eq7'}
\begin{aligned}
\frac{\partial u_z(t, x, y)}{\partial t} & =\mathcal{L}u_z(t, x, y)+\beta (u_z(t, x, y)-1)^2,\\
u_z(0,x;y) & =\begin{cases}
z, \ x=y,\\
1, \ x\ne y.
\end{cases}
\end{aligned}
\end{equation}
We define the factorial moments as
\[
m_k(t,x,y) := \Expect\prod_{i=0}^{k-1}(n(t,x,y)-i)=\frac{\partial^ku_z(t, x, y)}{\partial z^k}\bigg|_{z=1}.
\]
Then, differentiating the equation \eqref{l_eq7'} with respect to $z$ and substituting $z=1$, we obtain
\begin{align*}
\frac{\partial m_1(t,x,y)}{\partial t} & =\mathcal{L}m_1(t,x,y), \\
 m_1(0,x,y) & = \delta(y-x).
\end{align*}
Note that the Cauchy problems for $p(t,x,y)$ \eqref{eq2} and $m_1(t,x,y)$ are the same, so we get
\[
m_1(t,x,y) =p(t,x,y)
\]
and as a result
\[
\Expect n(t,y)=\sum\limits_{x\in \mathbb{Z}^{d}}\Expect n(t,x,y)=\sum\limits_{x\in \mathbb{Z}^{d}}m_1(t,x,y)=\sum\limits_{x\in \mathbb{Z}^{d}}p(t,x,y)\equiv 1
\]
(conservation law for the density of the total population).

The analysis of the second and higher moments ($k\ge3$) for the total population and subpopulations is more complex, see details in~\cite{Molchanov Whitmeyer}.

For the second moment $m_2(t,x,y)=\Expect [n(t,x,y)(n(t,x,y)-1)]$ we have the following results:
\begin{align*}
\frac{\partial m_2(t,x,y)}{\partial t} & =\mathcal{L}m_2(t,x,y)+2\beta m_1^2(t,x,y), \\
m_2(0,x,y) & = 0.
\end{align*}
We consider the Fourier representations
\begin{align*}
\mathcal{F}(\mathcal{ L}(m_1(t,x,y))) &=\mathcal{\hat L}(\hat m_1(t,k,y))=\kappa\hat m_1(t,k,y)(\hat a(k)-1),\\
\mathcal{\hat L}(k)&=\kappa(\hat a(k)-1).
\end{align*}
Thus, one can obtain
\begin{align*}
\frac{\partial \hat m_1(t,x,y)}{\partial t} & =\mathcal{\hat L}(k)\hat m_1(t,k,y), \\
\hat m_1(0,k,y)  & =\sum_x \delta(y-x)e^{i(k,x)}=e^{i(k,y)},
\end{align*}
its solution takes the form
\[
\hat m_1(t,k,y)  =e^{t\mathcal{\hat L}(k)}e^{i(k,y)},\\
\]
and for
\begin{align*}
\frac{\partial \hat m_2(t,k,y)}{\partial t} & =\mathcal{\hat L}(k)\hat m_2(t,k,y)+2\beta \hat m_1(t,k,y)\ast \hat m_1(t,k,y), \\
\hat m_2(0,k,y) & = 0,
\end{align*}
we obtain
\[
\hat m_2(t,k,y) e^{-t\mathcal{\hat L}(k)}   = 2\beta \int_0^t \hat m_1(t,k,y)\ast \hat m_1(t,k,y)e^{-s\mathcal{\hat L}(k)}\, ds,
\]
where ``$\ast$'' denotes the convolution operator.

Denote
\[
m_k(t,x,\Gamma) := \sum\limits_{y\in\Gamma}m_k(t,x,y)
\]
and note that
\begin{align*}
\sum_{x\in \mathbb{Z}}m_2(t,x,\Gamma) & =\hat m_2(t,k,\Gamma)\bigg|_{k=0}=2\beta \int_0^t\int_{T^d} \hat m_1(s,-\theta,\Gamma) \hat m_1(s,\theta,\Gamma)\,d\theta\, ds\\
&= 2\beta \int_{T^d}  d\theta \int_0^t \sum\limits_{y_1\in \Gamma}  \sum\limits_{y_2\in \Gamma} \cos((\theta, y_2-y_1)) e^{2s\mathcal{\hat L}(\theta)}\,ds\\
& = \beta \int_{T^d} \frac{(e^{2t\mathcal{\hat L}(\theta)}-1)(\sum\limits_{y_1\in \Gamma}  \sum\limits_{y_2\in \Gamma} \cos((\theta, y_2-y_1)))}{\mathcal{\hat L}(\theta)}\,d\theta.
\end{align*}
If the random walk generated by the operator $\mathcal{L}$ is recurrent, i.e.
\begin{align*}
G_0(0,0)&=\int\limits_0^{\infty}p(t,0)dt=\int\limits_0^{\infty}dt\frac{1}{(2\pi)^d}\int\limits_{T^d}e^{-\kappa t(1-\hat a(k))}\,dk\\
& =\frac{1}{(2\pi)^d}\int\limits_{T^d}\frac{dk}{\kappa(1-\hat a(k))}=\infty,
\end{align*}
then for $t\to\infty$
\begin{align*}
&m_2(t,\Gamma)\sim \hat m_2(t,0,\Gamma)\to \beta \int_{T^d}\frac{\sum\limits_{y_1\in \Gamma}  \sum\limits_{y_2\in \Gamma} \cos((\theta, y_2-y_1))\,d\theta}{-\mathcal{\hat L}(\theta)}\to \infty,\\
& m_2(t,y)\sim \beta (2\pi)^d G_0(0,0) \to \infty.
\end{align*}
At the physical level this is a manifestation of the high irregularity of the field $n(t,y)$. For large $t$ such a field (in which the density is meant to be 1!) consists of big islands separated by large distances (the phenomenon which has different names: clusterization, clumping, intermittency, see \cite{Zeldovich,Gartner}). Our goal is to explain this phenomenon.

Note that the recurrence of $X(t)$ with a finite variance of jumps can appear only in dimensions $d=1$ and  $d=2$, and implies that
\[
\int\limits_{T^d}\frac{dk}{1-\hat a(k)}=\infty.
\]

If, however, the process $X(t)$ is transient, i.e
\[
\int\limits_{T^d}\frac{dk}{1-\hat a(k)}<\infty,
\]
then field $n(t,y)$ converges in law to a steady state (statistical equilibrium), see~\cite{Molchanov Whitmeyer}.

On the other hand, if we consider the sums
\[
\sum\limits_{y\in \mathbb{Z}^{d}}n(t,x,y)=n_x(t),
\]
i.e. the total number of particles in the subpopulation,
then they are independent Galton-Watson processes and the generating function
\[
\varphi(t,z)=\Expect z^{n_x(t)}
\]
is the solution of the differential equation
\begin{equation}\label{generating_f}
\begin{aligned}
\frac{\partial \varphi(t,z)}{\partial t} & = \beta(\varphi(t,z)-1)^2, \\
\varphi(0,z) & = z.
\end{aligned}
\end{equation}

\begin{proposition}
\begin{equation}\label{prop_1}
\Expect [n_x(t)|n_x(t)>0] = \beta t+1.
\end{equation}
\end{proposition}
\begin{proof}
By integrating \eqref{generating_f} we obtain
\begin{equation}\label{eq_GS}
\varphi(t,z) =1-\frac{1-z}{(\beta t+1)-\beta t z}.
\end{equation}
Therefore,
\begin{align}\label{P_n_x(t)=0}
P\{n_x(t)=0\} &=\varphi(t,z)\bigg|_{z=0}=1-\frac{1}{\beta t+1}\quad
\text{for}~ k\ge 1,\\
\notag
P\{n_x(t)=k\}&=\frac{1}{k!}\frac{\partial\varphi^k (t,z)}{\partial z^k}\bigg|_{z=0}=\frac{(\beta t)^{k-1}}{(\beta t+1)^{k+1}}\quad \text{for}~ k\ge 1.
\end{align}

Then
\begin{multline*}
\Expect [n_x(t)|n_x(t)>0]=\sum\limits_{k=1}^{\infty}k P\{n_x(t)=k|n_x(t)>0\} \\=\frac{\sum\limits_{k=1}^{\infty} k P\{n_x(t)=k\}}{P\{n_x(t)>0\}}
\frac{\Expect [n_x(t)]}{P\{n_x(t)>0\}}=\beta t+1.
\end{multline*}
\end{proof}

\begin{theorem}
If $t\to\infty$ and $s>0$ then
\[
P\bigg\{\frac{n_x(t)}{\beta t+1}>s\bigg|n_x(t)>0\bigg\}\to e^{-s}.
\]
\end{theorem}
\begin{proof}
\begin{equation}
\begin{aligned}
&P\bigg\{\frac{n_x(t)}{\beta t+1}>s\bigg|n_x(t)>0\bigg\}=\frac{P\{n_x(t)>(\beta t+1)s,n_x(t)>0\}}{P\{n_x(t)>0\}}\\
&=\frac{\frac{1}{\beta t+1}-\sum\limits_{k=1}^{(\beta t+1)s}\frac{(\beta t)^{k-1}}{(\beta t+1)^{k+1}}}{\frac{1}{\beta t +1}}=1-\frac{1}{\beta t}\sum\limits_{k=1}^{(\beta t+1)s}\bigg(\frac{\beta t}{\beta t+1}\bigg)^k\\
& =1-\frac{1-\bigg(\frac{\beta t}{\beta t+1}\bigg)^{(\beta t+1)s}}{\bigg(\beta t+1\bigg)\bigg(1-\frac{\beta t}{\beta t+1}\bigg)}\to 1-(1-e^{-s})=e^{-s}.
\end{aligned}
\end{equation}
\end{proof}

We have the following situation: the majority of the subpopulations $n_x(t)$ are vanishing on the large time interval $[0,t]$ but for the remaining subpopulations, proportion of which is $\sim (\beta t)^{-1}$, the number of particles will have the order $O(t)$ at least at the level of the first moment. What is the structure of such large conditional subpopulations? We have to answer this question if we want to understand why vanishing of majority of the subpopulations does not affect the convergence to the steady state in the recurrent case $G_0(0,0)=\infty$ and leads to clusterization.

\section{The structure of the subpopulation $n(t,x,y)$, $y\in \mathbb{Z}^{d}$, for fixed $x$, large $t$, and $n_x(t)>0$}
\label{sec:main}

The joint generating function of the processes $n_x(t)$ and $n(t,x,y)$, i.e.
\[
u_{z, z_1}(t,x,y)=\Expect z^{n_x(t)}z_1^{n(t,x,y)}
\]
is the solution of the same equation \eqref{l_eqr} but with the different critical condition:
\begin{equation}
\label{l_eq9}\begin{aligned}
\frac{\partial u_{z, z_1}(t,x,y)}{\partial t} &=\mathcal{L}u_{z, z_1}(t,x,y)+\beta (u_{z, z_1}(t,x,y)-1)^2,\\
u_{z, z_1}(0,x,y) &=\begin{cases}
z, \ x\ne y\\
z z_1, \ x=y.
\end{cases}
\end{aligned}
\end{equation}

We cannot solve the non-linear equation \eqref{l_eq9} but using the differentiation over $z_1$ and substitution $z_1=1$ one can calculate or estimate the conditional moments
$\Expect [n(t,x,y)|n_x(t)>0]$, $\Expect [n(t,x,y)(n(t,x,y)-1)|n_x(t)>0]$ etc.

Put
\[
\tilde m_1(t,x,y,z) =\frac{\partial u_{z,z_1}(t,x,y)}{\partial z_1}\bigg|_{z_1=1}=\Expect_x z^{n_x(t)}n(t,x,y).
\]
Denote $\gamma(t,z):=2\beta(u_{z,z_1=1}(t,x,y)-1)$, then
\begin{equation}\label{l_eq11}
\begin{aligned}
\frac{\partial \tilde m_1(t,x,y,z)}{\partial t} & =\mathcal{L}\tilde m_1(t,x,y,z)+\gamma(t,z)\tilde m_1(t,x,y,z),\\
\tilde m_1(0,x,y,z) & =z\delta(y-x).
\end{aligned}
\end{equation}
Note that due to \eqref{eq_GS} we have
\[
\gamma(t,z)=-2\beta\frac{1-z}{1+\beta t-\beta tz}.
\]

\begin{lemma}
Consider the parabolic problem with the potential, depending only on the time $t$:
\begin{align*}
\frac{\partial u(t,x)}{\partial t} & =\mathcal{L}u(t,x)+\gamma(t)u(t,x),\\
u(0,x)  & =\varphi(x).
\end{align*}
Then
\[
u(t,x)=e^{\int\limits_0^t \gamma(s)\,ds}\sum\limits_v p(t,x-v)\varphi(v).
\]
\end{lemma}
\begin{proof}
Consider
\[
u(t,x)=e^{\int\limits_0^t \gamma(s)\,ds}v(t,x),
\]
then
\begin{align*}
\ \frac{\partial v(t,x)}{\partial t} & =\mathcal{L}v(t,x), \\
v(0,x) & =\varphi(x).
\end{align*}
Taking a Fourier transform gives:
\begin{align*}
\frac{\partial \hat{v}(t,k)}{\partial t} & =\kappa(\hat{a}(k)-1)\hat{v}(t,k), \\
\hat{v}(0,k) & =\hat{\varphi}(k).
\end{align*}
Note that
\[
\hat{p}(t,k)=\sum\limits_x p(t,0,x)e^{i(k,x)}=e^{-t\kappa(1-\hat{a}(k))},
\]
then
\begin{align*}
\hat{v}(t,k)&=e^{-t \kappa(\hat{a}(k)-1)}\hat{\varphi}(k)=\hat{p}(t,k)\hat{\varphi}(k),\\
v(t,x)&=\sum\limits_z p(t,x-z)\varphi(z),
\end{align*}
Therefore, we obtain
\[
u(t,x)=e^{\int\limits_0^t \gamma(s)\,ds}\sum\limits_v p(t,x-v)\varphi(v).
\]
\end{proof}

\begin{lemma}\label{lemma_2_2}
\begin{equation}\label{l_eq15}
\Expect [n(t,x,y)|n_x(t)=k]=kp(t,x,y).
\end{equation}
\end{lemma}
\begin{proof}
Let $X,Y$ be two integer-valued nonnegative random variables then
\begin{align*}
P_{k,l}&=P\{X=k,Y=l\},\\
\Phi(z,z_1)&=\Expect z^X z_1^Y,\\
\frac{\partial \Phi}{\partial z_1}\bigg|_{z_1=1}&=\Expect z^XY=\tilde m_1(z).
\end{align*}
Then
\begin{align*}
\tilde m_1(z)&=\sum\limits_{k=0}^\infty z^k\sum\limits_{l=0}^\infty P_{k,l}\cdot l=\sum\limits_{k=0}^\infty z^k \sum\limits_{l=0}^\infty l\cdot P\{l|X=k\}P\{X=k\}\\
& =\sum\limits_{k=0}^\infty z^kP\{X=k\}\Expect[Y|X=k].
\end{align*}
Let us return to the equation \eqref{l_eq11}. Note that
\[
e^{\int\limits_0^t\gamma(s)\,ds}=e^{-2\beta\int\limits_0^t\frac{(1-z)\,ds}{\beta s(1-z)+1}}=e^{-2\ln(\beta t(1-z)+1)}=\frac{1}{(\beta t+1-\beta t z)^2},
\]
then
\begin{multline*}
\tilde m_1(t,x;z,y)=\Expect(z^{n_x(t)}n(t,x,y))=\frac{zp(t,y-x)}{(\beta t+1-\beta t z)^2}=\frac{z p(t,y-x)}{(\beta t+1)^2} \\
\times \sum\limits_{k=1}^\infty k\bigg(\frac{\beta t z}{\beta t+1}\bigg)^{k-1}=\sum\limits_{k=1}^\infty \frac{k(\beta t)^{k-1}}{(\beta t+1)^{k+1}}p(t,x,y)z^k.
\end{multline*}
We used here the expansion of $\bigg(1-\frac{\beta t z}{\beta t+1}\bigg)^{-2}$ into the Taylor series. Since
\[
\frac{(\beta t)^{k-1}}{(\beta t+1)^{k+1}}=P\{n_x(t)=k\}
\]
we have finally that
\[
\Expect [n(t,x,y)|n_x(t)=k]=kp(t,x,y).
\]
\end{proof}
\begin{theorem} For $t\to \infty$ we have
\[
\Expect [n(t,x,y)|n_x(t)>0]\sim \beta tp(t,x,y).
\]
\end{theorem}
\begin{proof}
\begin{align*}
& \Expect[n(t,x,y)|n_x(t)>0]=\sum\limits_{l=1}^{\infty}lP\{n(t,x,y)=l|n_x(t)>0\}\\
&=\sum\limits_{l=1}^{\infty}l\frac{P\{n(t,x,y)=l,n_x(t)>0\}}{P\{n_x(t)>0\}}=\frac{\sum\limits_{l=1}^{\infty}lP\{n(t,x,y)=l,n_x(t)>0\}}{\sum\limits_{k=1}^{\infty}P\{n_x(t)=k\}}\\
&=\frac{\sum\limits_{l=1}^{\infty}\sum\limits_{k=1}^{\infty} l P\{n(t,x,y)=l,n_x(t)=k\}}{\sum\limits_{k=1}^{\infty}P\{n_x(t)=k\}}\\
&=\frac{\sum\limits_{l=1}^{\infty}\sum\limits_{k=1}^{\infty} l P\{n_x(t)=k\} P\{n(t,x,y)=l|n_x(t)=k\}}{\sum\limits_{k=1}^{\infty}P\{n_x(t)=k\}}\\
& =\frac{\sum\limits_{k=1}^{\infty}P\{n_x(t)=k\}\Expect[n(t,x,y)|n_x(t)=k]}{\sum\limits_{k=1}^{\infty}P\{n_x(t)=k\}}.
\end{align*}
Then using \eqref{l_eq1} we have
\begin{align*}
& \Expect [n(t,x,y)|n_x(t)>0]=\frac{p(t,x,y)\sum\limits_{k=1}^{\infty}k\frac{(\beta t)^{k-1}}{(\beta t+1)^{k+1}}}{\frac{1}{\beta t+1}}\\
& =\frac{p(t,x,y)}{\beta t+1}\sum\limits_{k=1}^{\infty}k\bigg(\frac{\beta t}{\beta t+1}\bigg)^{k-1}=\frac{p(t,x,y)}{(\beta t+1)(1-\frac{\beta t}{\beta t+1})^2}\\
& =(\beta t+1)p(t,x,y)\sim \beta tp(t,x,y).
\end{align*}
\end{proof}

Then for $|x-y|=O(\sqrt{t})$
\[
\Expect [n(t,x,y)|n_x(t)>0]=\begin{cases}
\sqrt{t}, \ d= 1;\\
1, \ d=2;\\
\frac{1}{t^{d/2-1}}, \ d\ge 3.
\end{cases}
\]

The formula \eqref{l_eq15} is not difficult to understand without calculations. Each particle, among $n_x(t)=k$ particles, performs on $[0,t]$ a random walk with the generator $\mathcal{L}$ (which is the union of the pieces of paths between the successive transformations). These paths are highly dependent, but for calculating the first moment this dependence is irrelevant.

For the second (factorial) moment we can use the same approach. Differentiating \eqref{l_eq9} two times over $z_1$ and substituting $z_1=1$ we get for
\[
\tilde m_2(t,x,y,z)=\Expect_x z^{n_x(t)}[n(t,x,y)(n(t,x,y)-1)]
\]
the equation
\begin{align*}
 \frac{\partial \tilde m_2(t,x,y,z)}{\partial t}&=\mathcal{L}\tilde m_2(t,x,y,z)+\gamma(t,z)\tilde m_2(t,x,y,z)+2\beta\tilde m_1^2(t,x,y,z),\\
 \tilde m_2(0,x,y,z)&=0.
\end{align*}
Here
\[
\gamma(t,z)=-2\beta(1-u_{z,z_1})|_{z_1=1}=-\frac{2\beta(1-z)}{(\beta t+1)-\beta t z}
\]
is the same space independent potential as in $\tilde m_1(t,x,y,z)$, see above.

\begin{lemma}
Consider the non-homogeneous parabolic problem
\begin{align*}
\frac{\partial u(t,x)}{\partial t} & =\mathcal{L}u(t,x)+\gamma(t)u(t,x)+f(t,x),\\
u(0,x) & =0.
\end{align*}
Then the substitution $u=e^{\int\limits_0^t\gamma(s)\,ds} v(t,x)$ leads to the equation
\begin{equation}\label{l_eq18}
\begin{aligned}
\frac{\partial v(t,x)}{\partial t}&=\mathcal{L}v(t,x)+\tilde f(t,x),\quad \tilde{f}(t,x)=e^{-\int\limits_0^t\gamma(s)\,ds} f(t,x),\\
v(0,x)&=0
\end{aligned}
\end{equation}
and its solution is given by Duhamel's formula
\begin{equation}\label{l_eq19}
v(t,x)=\int\limits_0^t \bigg[\sum\limits_z p(t-s,x-z)\tilde{f}(s,z)\bigg]\,ds.
\end{equation}
\begin{proof}
Taking a Fourier transform gives:
\begin{align*}
\frac{\partial \hat{v}(t,k)}{\partial t}&=\kappa(\hat{a}(k)-1)\hat{v}(t,k)+\hat{\tilde{f}}(t,k),\\
\hat{v}(t,k)&=e^{t\kappa(\hat{a}(k)-1)}h(t),\ \frac{\partial h(t,k)}{\partial t}=\hat{\tilde{f}}(t,k)e^{-t\kappa(\hat{a}(k)-1)},\\
h(t,k)&=\int\limits_0^t \hat{\tilde{f}}(s,k)e^{-s\kappa(\hat{a}(k)-1)}\,ds,\\
\hat{v}(t,k)&= e^{t\kappa(\hat{a}(k)-1)} \int\limits_0^t \hat{\tilde{f}}(s,k)e^{-s\kappa(\hat{a}(k)-1)}\,ds=\int\limits_0^t e^{(s-t)\kappa(1-\hat{a}(k))}\hat{\tilde{f}}(s,k)\,ds.
\end{align*}

Note that
\[
\hat{p}(t-s,k)=\sum\limits_x p(t-s,0,x)e^{i(k,x)}=e^{(s-t)\kappa(1-\hat{a}(k))}
\]
then
\begin{align*}
\hat{v}(t,k)&=\int\limits_0^t \hat{p}(t-s,k) \hat{\tilde{f}}(s,k)\,ds,\\
v(t,x)&=\int\limits_0^t \bigg[\sum\limits_z p(t-s,x-z)\tilde{f}(s,z)\bigg]\,ds.
\end{align*}
\end{proof}
\end{lemma}

Combining together formulas \eqref{l_eq18}, \eqref{l_eq19} and expressions
\[
\tilde m_1(t,x,y,z)=\frac{z p(t,x,y)}{(\beta t+1-\beta t z)^2},\quad e^{\int\limits_0^t\gamma(s)\,ds}={(\beta t+1-\beta t z)^{-2}}
\]
we can find the conditional second moment
\[
\tilde m_2(t,x,y,z)=2\beta\int\limits_0^t \,ds \frac{\sum\limits_{v} z^2 p^2(s, v-y)p(t-s,x-v)}{(\beta s+1-\beta s z)^2(\beta t+1-\beta t z)^2}
\]

Let us now represent the integrand by the Taylor series. Since
\[
\varphi(z)=\frac{1}{(\beta t+1-\beta t z)^2}=\sum\limits_{k=0}^\infty (k+1)\frac{(\beta t)^{k-1}}{(\beta t+1)^{k+1}}z^k
\]
we will get
\begin{multline*}
 \frac{1}{(\beta t+1-\beta t z)^2}\cdot\frac{1}{(\beta s+1-\beta z s)^2}\\
 =\sum\limits_{n=0}^{\infty}z^n\bigg[\sum\limits_{l=0}^n (n-l+1)(l+1)\frac{(\beta t)^{n-l}}{(1+\beta t)^{n-l+2}}\cdot\frac{(\beta s)^l}{(1+\beta s)^{l+2}}\bigg]
\end{multline*}
i.e. (setting $m=n+2$) we will get finally
\begin{multline*}
\tilde m_2(t,x,y,z)=2\beta \int\limits_0^t \,ds\bigg[\sum\limits_{v\in \mathbb{Z}^{d}}p^2(s,v-y)p(t-s, x-v)\cdot\sum\limits_{m=2}^\infty z^m\frac{(\beta t)^{m-1}}{(1+\beta t)^{m+1}}\\
\times \big[\sum\limits_{l=0}^{m-2} (m-l+1)(l+1)\bigg(\frac{\beta t}{\beta t+1}\bigg)^{-l-1}\frac{(\beta s)^l}{(\beta s+1)^{l+2}}\big]\bigg].
\end{multline*}

Due to lemma \ref{lemma_2_2} we have the following formula for the conditional second moment
\begin{equation}\label{l_w}
\begin{aligned}
M_2(m,t,x,y)&=\Expect_x[n(t,x,y)(n(t,x,y)-1)|n_x(t)=m]\\
& =2\beta\int\limits_0^t \,ds\bigg[\sum\limits_{v\in \mathbb{Z}^{d}}\frac{p^2(s,x-v)p(t-s, v-y)}{(\beta s+1)^2}\\
& \times\bigg(\frac{\beta t+1}{\beta t}\bigg)^{l+1}\sum\limits_{l=0}^{m-2}(m-l+1)(l+1)\bigg(\frac{\beta s}{\beta s+1}\bigg)^l\bigg]
\end{aligned}
\end{equation}

Like in the case of the first moment the most important values of $m$ must be of order $t$. To simplify the calculations we will put $m=t, \beta=1$. Transition to the more general case $m=ct, \beta \ne 1$ is simple.

We start from the rough estimation of the second moment $M_2(t,t,\cdot)$. The const $c$ in the calculations below will be ``universal'' , i.e. its meaning can be different in the neighbouring formulas. We start from several simple inequalities.
\begin{equation*}
\begin{aligned}
& a) \sum\limits_{v\in \mathbb{Z}^{d}}p^2(s,x-v)p(t-s, v-y)\le p(s,0)\sum\limits_v p(s,x-v)p(t-s,v-y)\\
& =p(s,0)p(t,x-y),\ \text{see formula \eqref{l_eq4};}\\
& b) \ p(s,0)\le \frac{c}{(1+s)^{d/2}}, \ s\ge 0,\\
& \text{this is the corollary of the local CLT, see \eqref{l_eq5};}\\
& c) \ \bigg(\frac{1+t}{t}\bigg)^l\le c,\ \ \text{if } l\le m= t;\\
& d) \sum\limits_{l=0}^{t-2}(m-l+1)(l+1)\bigg(\frac{s}{s+1}\bigg)^l\le (t+1)\sum\limits_{l=0}^{\infty}(l+1)\bigg(\frac{s}{s+1}\bigg)^l\\
& \le (t+1)\frac{1}{(1-\frac{s}{s+1})^{2}}=(t+1)(s+1)^2.
\end{aligned}
\end{equation*}

Then formula \eqref{l_w} together with four inequalities above for all $x, y$ gives
\begin{align*}
M_2(t,x,y)&:=M_2(m=t,t,x,y)=\Expect [n(t,x,y)(n(t,x,y)-1)|n_x(t)=t]\\
& \le cp(t,x,y)(t+1)\int\limits_0^t\frac{ds (s+1)^2}{(1+s)^{d/2+2}}\le \frac{c}{t^{d/2-1}}\int\limits_0^t\frac{ds}{(1+s)^{d/2}}
\end{align*}
It gives
\begin{equation}
\label{M2inequalities}
\begin{aligned}
M_2(t,x,y)&\le ct \text{ in dimension } d=1;\\
M_2(t,x,y)&\le c\ln{t} \text{ for } d=2;\\
M_2(t,x,y)&\le \frac{c}{t^{d/2-1}} \text{ for } d\ge3 \ (=\frac{c}{\sqrt t} \text{ for } d=3).
\end{aligned}
\end{equation}
Assume $a(z)$ is bounded with the following asymptotic expansion at infinity:
\begin{equation}\label{a(z)}
a(z)=\sum\limits_{j=0}^N\frac{c_j(\dot{z})}{|z|^{d+\alpha_j}}+O(|z|^{-2d-\alpha-l}), \quad z\to\infty,
\end{equation}
where $\alpha>2$, $\alpha_0=\alpha<\alpha_1<\ldots<\alpha_N$, $\dot{z}=z/|z|$, and $c_0(\dot{z})=c_0(\dot{-z})$ is a positive continuous function on $\mathcal{S}^{d-1}$, $l=1$ if $\alpha>[\alpha]$, $l=2$ if $\alpha=[\alpha]$, $c_0(\dot{z})>0$, and $c_j(\dot{z})$ are sufficiently smooth.
\begin{theorem}
Let condition \eqref{a(z)} be satisfied, for $|x-y|<C\sqrt{t}$ and $t\to\infty$
\begin{equation}\label{th_2_5}
\begin{aligned}
M_2(t,x,y)&\sim C_1t \text{ for } d=1;\\
M_2(t,x,y)&\sim C_2\ln{t} \text{ for } d=2;\\
M_2(t,x,y)&\sim \frac{C_d}{t^{\frac{d}{2}-1}}  \text{ for } d\ge3.
\end{aligned}
\end{equation}
\end{theorem}
\begin{proof}
From (\cite{Molchanov Whitmeyer 2018}, Theorem 1.1) there are constants $A$, $\epsilon>0$,
such that the solution $p(t,z)$ has the following asymptotic behavior as $|z|^2\ge At$, $t\ge 0$, $x\to\infty$:
\[
p(t,z)=\frac{t}{|z|^{d+\alpha}}\left[c_0(\dot{z})+O\left(\left(\frac{1+t}{|z|^2}\right)^{\epsilon}\right)\right]+\frac{e^{-\frac{|z|^2}{2t}}}{(2\pi t)^{d/2}}\left(1+O\left(\frac{t^{1/\alpha}}{|z|}\right)\right),\\
\]
and from the local CLT
\[
p(t,z)\sim \frac{e^{-\frac{|z|^2}{2t}}}{(2\pi t)^{d/2}}, \ |z|<C\sqrt{t}, \ t\to\infty,
\]
then
\begin{align*}
& \sum\limits_v p^2(s,x-v)p(t-s,v-y) = \sum\limits_{v:\{|x-v|<C\sqrt{t}\}} p^2(s,x-v)p(t-s,v-y)\\
& +\sum\limits_{v:\{|x-v|\ge C\sqrt{t}\}} p^2(s,x-v)p(t-s,v-y)\\
& \ge \sum\limits_{v:\{|x-v|<C\sqrt{t}\}} \frac{1}{{(4\pi s)^{d/2}}} \cdot \frac{e^{-\frac{(x-v)^2}{2(\frac{s}{2})}}}{(2\pi (\frac{s}{2}) )^{d/2}}\cdot p(t-s,v-y)\\
&+ \sum\limits_{v:\{|x-v|\ge C\sqrt{t}\}} \frac{p(\frac{s}{2},x-v)p(t-s,v-y)}{{(4\pi s)^{d/2}}} \\
& -\sum\limits_{v:\{|x-v|\ge C\sqrt{t}\}} \frac{p(\frac{s}{2},x-v)p(t-s,v-y)}{{(4\pi s)^{d/2}}}= \sum\limits_v \frac{p(\frac{s}{2},x-v)p(t-s,v-y)}{{(4\pi s)^{d/2}}}\\
& -\sum\limits_{v:\{|x-v|\ge C\sqrt{t}\}} \frac{p(\frac{s}{2},x-v)p(t-s,v-y)}{{(4\pi s)^{d/2}}}=\frac{p(t-\frac{s}{2},x-y)}{{(4\pi s)^{d/2}}}-
\end{align*}
\begin{align*}
&-\sum\limits_{v:\{|x-v|\ge C\sqrt{t}\}} \frac{p(\frac{s}{2},x-v)p(t-s,v-y)}{{(4\pi s)^{d/2}}}\frac{p(t-\frac{s}{2},x-y)}{{(4\pi s)^{d/2}}}\\
&-\sum\limits_{v:\{|x-v|\ge C\sqrt{t}\}} \frac{1}{(4\pi s)^{d/2}}\cdot\bigg(\frac{s c_0(\dot{x-v})}{2|x-v|^{d+\alpha}}+\frac{e^{-\frac{|x-v|^2}{2t}}}{(\pi s)^{d/2}}\bigg)\sim\\
& \sim \frac{e^{-\frac{(x-y)^2}{2(t-\frac{s}{2})}}}{(4\pi s )^{d/2}(t-s/2)^{d/2}}>\frac{e^{-1}}{(4\pi s )^{d/2}(t-s/2)^{d/2}}  \ \text{ for }t \to \infty, \ s\in [0,t).
\end{align*}
Due to \eqref{l_w} we obtain
\begin{align*}
 M_2(m,t,x,y)&=\Expect_x[n(t,x,y)(n(t,x,y)-1)|n_x(t)=m]\\
&= 2\beta \bigg(\frac{\beta t+1}{\beta t}\bigg)\cdot \int\limits_0^t \,ds \bigg[ \sum\limits_{v\in \mathbb{Z}^{d}}p^2(s,x-v)p(t-s, v-y)\cdot \frac{1}{(\beta s+1)^2}\\
& \times\sum\limits_{l=0}^{m-2}(-l^2+ml+m+1)\bigg(\frac{s (\beta t+1)}{t (\beta s+1)}\bigg)^l\bigg]
\end{align*}
Denote $q=\frac{s (\beta t+1)}{t (\beta s+1)}$, then for $m=t$ and $\beta=1$
\begin{align*}
& \sum\limits_{l=0}^{m-2}(-l^2+ml+m+1)\bigg(\frac{s (\beta t+1)}{t (\beta s+1)}\bigg)^l=-\frac{q^2+q}{(1-q)^3}+\frac{mq}{(1-q)^2}+\frac{m+1}{1-q}\\
&=\frac{mq}{(1-q)^2}+\frac{1+m+mq+mq^2}{(1-q)^3}= \frac{m(q^2+q+1)+1-q}{(1-q)^2(1+q+q^2)}\sim \frac{m}{(1-q)^2}\\
& \sim \frac{t^3 (s+1)^2}{(t-s)^2},
\end{align*}
then
\[
 M_2(t,x,y)> c_d \int\limits_0^t \frac{t^3}{(2t-s)^{d/2}(s)^{d/2}(t-s)^2}\,ds.
\]
For $d=1$
\begin{align*}
M_2(t,x,y)&> c_1 \int\limits_0^t \frac{t^3}{\sqrt{2t-s}\sqrt{s}(t-s)^2}\,ds>c_1\int\limits_0^{t/2} \frac{t^3}{\sqrt{2t-s}\sqrt{s}(t-s)^2}\,ds\\
&=c_1\frac{t\sqrt{s}\sqrt{2t-s}}{t-s}\bigg|_0^{t/2} \sim \sqrt{3}c_1t= C_1t,
\end{align*}
for $d=2$
\begin{align*}
M_2(t,x,y)&>c_2\int\limits_0^t \frac{t^3}{(2t-s)s(t-s)^2}\,ds>c_2\int\limits_1^{t/2} \frac{t^3}{(2t-s)s(t-s)^2}\,ds\\
& =c_2\bigg(\frac{t}{t-s}-\frac{\ln(2t-s)}{2}+\frac{\ln s}{2}\bigg|_1^{t/2}\bigg)\\
& =c_2\bigg(2-\frac{\ln(3t/2)}{2}+\frac{\ln(t/2)}{2}-\frac{t}{t-1}+\frac{\ln(2t-1)}{2}\bigg)= C_2 \ln t,
\end{align*}
for $d\ge 3$ and constants $q_i (i=0,\ldots,d-1$) and $r_1, r_2$ the integrand $M_2(t,x,y)$ has following form: for $d=2n+1$
\begin{align*}
& M_2(t,x,y)> c_d \int\limits_0^t \frac{t^3}{(2t-s)^{d/2}(s)^{d/2}(t-s)^2}\,ds\\
& > \frac{\sum\limits_{k=0}^{2n} q_k t^{2n}s^{2n-k}}{s^{n-\frac{1}{2}}(2t-s)^{n-\frac{1}{2}}(t-s)t^{2n-1}}\bigg|_1^{t/2}\sim \frac{c_1 t^{2n}}{t^{4n-1}} + \frac{c_2 t^{2n}}{t^{3n-\frac{1}{2}}} \sim\frac{c}{t^{n-\frac{1}{2}}}=\frac{C_d}{t^{\frac{d}{2}-1}}
\end{align*}
and for $d=2n$
\begin{align*}
& M_2(t,x,y)> c_d \int\limits_0^t \frac{t^3}{(2t-s)^{d/2}(s)^{d/2}(t-s)^2}\,ds>\frac{\sum\limits_{k=0}^{2n-1} q_k t^{2n-1}s^{2n-1-k}}{s^{n-1}(2t-s)^{n-1}(t-s)t^{2n-2}}\\
& +\frac{r_1 \ln(2t-s)+ r_2 \ln s}{t^{2n-2}}\bigg|_1^{t/2}\sim \frac{c_1 t^{2n-1}}{t^{4n-3}} + \frac{c_2 t^{2n-1}}{t^{3n-2}}+\frac{c_3 \ln t}{t^{2n-2}} \sim\frac{c}{t^{n-1}}=\frac{C_d}{t^{\frac{d}{2}-1}}.
\end{align*}
Therefore, one can combine this inequalities with \eqref{M2inequalities} to obtain \eqref{th_2_5}.
\end{proof}

\section{Conclusion}
\label{sec:conclusion}

Assuming that $n(0,x)\equiv1$ for the large moment $t$ the majority of the subpopulation $n(t,x,\cdot)$ will degenerate, see \eqref{prop_1} and \eqref{P_n_x(t)=0}. The initial points of all such (non-degenerated) subpopulations form the Bernoulli point field with parameter $p=\frac{1}{1+\beta t}$.

In a neighborhood of each such point the subpopulation forms the island $n(t,x,\cdot)$ containing the random number $n_{x_i}(t)$ of the particles with the mean value $\beta t+1$. Then Theorem 3.1 from \cite{Molchanov Whitmeyer 2018} gives
\[
\Expect n(t,x_{i,t},y)=\Expect [n(t,x_i,y)|n_x(t)>0]\sim\beta t p(t,x_i,y)\sim\frac{\beta t e^{-\frac{(y-x_i)^2}{2t}}}{(2\pi t )^{d/2}},
\]
i.e. the diameter of the subpopulation is $O(\sqrt{t})$.

Consider several cases: $d=1$, $d=2$ and $d\geq3$.
\paragraph{Case $d=1$.} Here the distances between the points $x_{i,t}$ have the geometric distribution with the mean value $\beta t$ and the subpopulations have typical size $\sqrt{t}$ (due to local CLT).

The total population demonstrates the high level of intermittency: large clusters, that is clusters of the diameter $O(\sqrt{t})$, are separated by the empty intervals of the length $O(t)$. Due to \eqref{prop_1} each of such clusters contains about $t$ particles.

\paragraph{Case $d=2$.} Here the typical size of the subpopulation and the typical distance between points $x_{i,t}$ both have order $\sqrt{t}$. However, the population still has fairly large gaps. Let us estimate these gaps.

Consider on the lattice $\mathbb{Z}^2$ the square $[-L, L]\times[-L,L]=Q_L$ with large enough $L$, containing $(2L+1)^2$ point. For $t=0$ any point $x\in Q_L$ contains one particle: at the moment $t\gg 1$ only $\frac{(2L+1)^2}{\beta t}=N_t$ subpopulations started from the points $x_{i,t}$ will survive.

Let us divide $Q_L$ into cells of the size ($\sqrt{t}a(t)\times\sqrt{t}a(t)$) where the function $a(t)$ will be selected later. Number of such cells is $\frac{(2L+1)^2}{t a^2(t)}=N_1(t)$. The probability that one cell does not contain points $x_{i,t}$ equals to
\[
\left(1-\frac{1}{\beta t +1}\right)^{t a^2(t)}=e^{-\frac{a^2(t)}{\beta}}.
\]
Then the mean number of empty cells in $Q_L$ will be
\[
\frac{4L^2}{ta^2(t)}e^{-\frac{a^2(t)}{\beta}}=\mu_t.
\]

Put $a^2(t)=\beta\ln t$. Then $a(t)\sim c\sqrt{\ln t}$ for large $t$.
Hence for $L=t\sqrt{\ln t}$ we obtain $\mu_t\sim\mathrm{const}$, i.e. empty cells have density $O(\frac{1}{t\sqrt{\ln t}})$, which means that in each cube of the size $t\sqrt{\ln t}\times t\sqrt{\ln t}$ there is a gap of the diameter $O(t\sqrt{\ln t})$. This is a very weak intermittency.

\paragraph{Case $d \ge 3$.} Here the population $n(t, x)$ is highly uniform.

\section{Simulation}
\label{sec:simulation}

We call the state of a BRW system the set of pairs $\{x, t\}$, each of which corresponds to a particle located at the point $x \in \mathbb{Z}^d$, that first appeared at the point $x$ at the time $t$. By evolution we mean a jump to another point, splitting or death of a particle. In the process of modeling the transition from one state of the BRW system to another will be carried out by excluding one pair from the set of states and adding to the seat of states of one or several pairs corresponding to the result of the simulated particle evolution.

\emph{Initialization.} First, we set the characteristics of the simulated BRW: choose the dimension $d$ of the integer lattice, the functions defining the distribution of the jumps matrix $A$ and the infinitesimal generating functions conditioned by the relation $\beta=\mu$, as well as the execution time $T$ and a sufficiently large part of the space in which we will consider the process. At the initial moment of time, the state of the system is determined by the presence of a single particle at each point of a given space.

\emph{Step of algorithm.} On input we have pairs $\{x, t\}$ waiting to be processed. We select an arbitrary pair $\{x, t\}$ (due to the independence of the particles) and model the evolution of the corresponding particle: the exponential time $dt$ of staying it in $x$, then the jump or birth/death event. In the case of a jump, the transition state is simulated according to the matrix $A$, in the next state the current pair $\{x, t\}$ disappears and appears new pair $\{x', t+dt\}$. In the case of death, the current pair $\{x, t\}$ disappears, and in the case of birth (dividing into two offsprings) the current pair $\{x, t\}$ disappears and two new pairs $\{x, t+dt\}$ are added to the set.

\emph{Stop condition.} The algorithm terminates when all the values of $t$ of all pairs in the system exceed the specified time $T$, or when the set of states becomes empty (the process has degenerated). Finally, we obtain the coordinates of all the points in the system remaining at time $T$ and plot them on the graph.

At a point $x \in \mathbb{Z}^d$ the particle can perform a symmetric random walk or branching, which is a Galton-Watson process with infinitesimal generating function $f(u)=\beta-2\beta u+\beta u^2$. In this case, the particle's behavior is described as follows: the particle spends at a point $x$ a random time, exponentially distributed with the parameter $1+2\beta$, and then goes to the point $x'$ with probability $a(x,x')/(1+2\beta)$ or die with probability $\beta/(1+2\beta)$ or divided into two offsprings with probability $\beta/(1+2\beta)$.

To simulate ``randomness'' in implementing the proposed algorithm, we use the following random number generators:
\begin{align*}
& \text{particle rest time }Exp(1+2\beta),\\
& \text{evolution }Discr('walk',\ \frac{1}{1+2\beta}; \ 'splitting/death',\ \frac{2\beta}{1+2\beta}).
\end{align*}
In the case of a walk, the transition probabilities from $x$ to $x'$ are given by
\[
Discr\bigg(\frac{a(x,x')}{1+2\beta}\ \bigg| \ x'\in \mathbb{Z}^d, \ x'\ne x\bigg)
\]
and in the case of splitting/death
\[
Discr('splitting',\ \frac{1}{2};\ 'death',\ \frac{1}{2}).
\]

For  $\mathbb{Z}^1$ the obtained clusters are plotted in Fig.~\ref{g1} and for $\mathbb{Z}^2$ they are plotted in Fig.~\ref{g2}.

\begin{figure}[htbp!]
\includegraphics[width=0.49\textwidth]{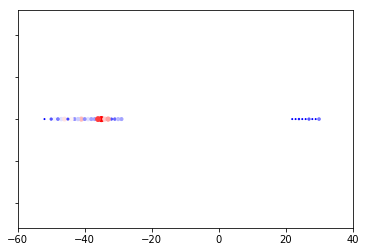}\qquad
\includegraphics[width=0.49\textwidth]{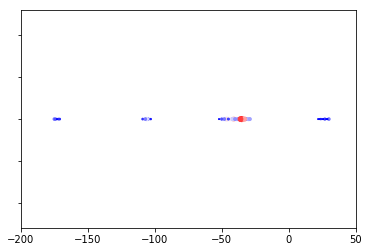}
\caption{Clusters on $\mathbb{Z}^1$}\label{g1}
\end{figure}

\begin{figure}[htbp!]
\includegraphics[width=0.49\textwidth]{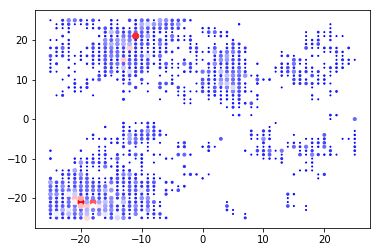}\qquad
\includegraphics[width=0.49\textwidth]{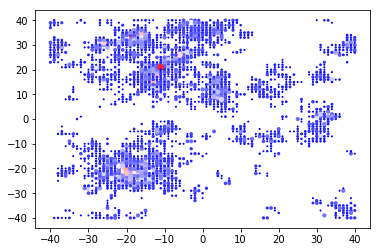}
\caption{Clusters on $\mathbb{Z}^2$}\label{g2}
\end{figure}

\textbf{Acknowledgements.} D.~Balashova and E.~Yarovaya were supported by the Russian Foundation for Basic Research (RFBR), project No. 17-01-00468. S.~Molchanov was supported by the Russian Science Foundation (RSF), project No. 17-11-01098.

\end{document}